\documentclass[10pt,preprint]{elsarticle}

\usepackage{amsmath}
\usepackage{amssymb}
\usepackage{amsthm}
\usepackage{bm}
\usepackage{mathtools}
\mathtoolsset{showonlyrefs=true}

\newcommand{\eset}[1]{\left\lbrace #1 \right\rbrace}
\newcommand{\iset}[2]{\left\lbrace #1 \ \middle|\  #2 \right\rbrace}
\DeclareMathOperator*{\argmin}{arg\,min}

\newcommand{\ratio}{q}
\newcommand{\norm}[2][p]{\| #2 \|_{#1}}
\newcommand{\getheight}{\rule{0em}{2.5ex}}

\newtheoremstyle{csutilitystyle}%
  {}%
  {}%
  {\normalfont}%
  {}%
  {\bfseries}%
  {. }%
  { }%
  {}%
\theoremstyle{csutilitystyle}

\newtheorem{theorem}{Theorem}
\newtheorem*{theorem*}{Theorem}

\newtheorem*{definition*}{Definition}
\newtheorem{lemma}[theorem]{Lemma}
\newtheorem*{lemma*}{Lemma}

\newtheorem*{corollary*}{Corollary}

\newtheorem*{proposition*}{Proposition}

\newtheorem*{problem*}{Problem}

\newtheorem*{Proof*}{Proof}
\renewenvironment{proof}{
\begin{Proof*}%
}{
\qed
\end{Proof*}%
}

\journal{a Journal}

\begin{document}

\begin{frontmatter}

\title{%
A generalization of the steepest-edge rule and \\
its number of simplex iterations for a nondegenerate LP
}

\author{Masaya Tano\corref{cor1}\fnref{label1}}
\ead{msytano@gmail.com}
\cortext[cor1]{Corresponding author. Tel.: +81-42-388-7451.}
\address[label1]{%
Graduate School of Engineering,\\
Tokyo University of Agriculture and Technology,\\
2-24-16 Naka-cho, Koganei-shi, Tokyo 184-8588, Japan\fnref{label1}}
\author{Ryuhei Miyashiro\fnref{label2}}
\address[label2]{%
Institute of Engineering,\\
Tokyo University of Agriculture and Technology,\\
2-24-16 Naka-cho, Koganei-shi, Tokyo 184-8588, Japan\fnref{label2}}
\author{Tomonari Kitahara\fnref{label3}}
\address[label3]{Department of Industrial Engineering and Economics,\\
Tokyo Institute of Technology, \\
2-12-1-W9-62 Ookayama, Meguro-ku, Tokyo 152-8550, Japan\fnref{label3}}

\begin{abstract}
In this paper, we propose a $p$-norm rule, which is a generalization of the steepest-edge rule, as a pivoting rule for the simplex method.
For a nondegenerate linear programming problem, we show upper bounds for the number of iterations of the simplex method with the steepest-edge and $p$-norm rules.
One of the upper bounds is given by a function of the number of variables, that of constraints, and the minimum and maximum positive elements in all basic feasible solutions.
\end{abstract}


\end{frontmatter}

\section{Introduction} \label{sec:Introduction}

The simplex method, an algorithm for a linear programming problem~(LP), has been improved since Dantzig~\cite{Dantzig-1963} developed it.
The primal simplex method starts from an initial basic feasible solution (BFS) and repeats exchanging a basic variable with a nonbasic variable until the solution satisfies the optimality condition.
A variable changed from a nonbasic to a basic variable is called an entering variable, while one for the other direction is called a leaving variable.

A pivoting rule is a criterion to choose entering and leaving variables.
The oldest rule is the most negative coefficient rule, which was developed by Dantzig.
It is well known that the simplex method with the most negative coefficient rule is not a polynomial-time algorithm~\cite{Klee-Minty-1972}.
Various pivoting rules have been proposed so far (cf.~\cite{Terlaky-Zhang-1993}), and they have great influence on the number of simplex iterations.

Nowadays there are pivoting rules that spend fewer simplex iterations in practice than the most negative coefficient rule.
The steepest-edge rule~\cite{Kuhn-Quandt-1963,Wolfe-Cutler-1963} is such a pivoting rule.
However, even for the steepest-edge rule, there is an LP in which the number of iterations is exponential with respect to the number of variables~\cite{Goldfarb-Sit-1979}.
It is a long-standing open question in linear programming whether there exists a pivoting rule or a variant of the simplex method that can solve any LP in polynomial time.

Some recent studies analyzed the simplex method from a different perspective.
Kitahara and Mizuno~\cite{Kitahara-Mizuno-2013-a} showed upper bounds for the number of different BFSs generated by the simplex method with the most negative coefficient rule.
These bounds are given as follows:
\begin{equation} \notag
  \left\lceil \frac{m\gamma}{\delta} \log\left( \frac{\getheight\bm{c}^{\top}\bm{x}^0-z^*}{\getheight\bm{c}^{\top}\bar{\bm{x}}-z^*} \right) \right\rceil \quad \mathrm{and}\quad
  (n-m) \left\lceil \frac{m\gamma}{\delta} \log\left(\frac{m\gamma}{\delta}\right) \right\rceil,
\end{equation}
where $m$ is the number of constraints, $n$ is that of variables, $\bm{x}^0$~is an initial BFS, $\bar{\bm{x}}$~is a BFS with the second smallest objective value, $z^*$~is the optimal value, and $\gamma$ and~$\delta$ are the maximum and minimum positive elements in all BFSs, respectively.
On the assumption that an LP is nondegenerate, these bounds yield upper bounds for the number of iterations of the simplex method with the most negative coefficient rule and that with the best improvement rule.

The contribution of this paper is to expand the research in~\cite{Kitahara-Mizuno-2013-a} into the steepest-edge rule.
We first propose a $p$-norm rule as a pivoting rule for the simplex method.
This rule is a generalized steepest-edge rule and coincides with the steepest-edge when $p=2$.
Next we prove that the simplex method with the $p$-norm rule finds an optimal solution in at most
\begin{equation} \notag
  \left\lceil m^{1+\frac{1}{p}} \frac{\gamma^2}{\delta^2} \log\left( \frac{\getheight\bm{c}^{\top}\bm{x}^0-z^*}{\getheight\bm{c}^{\top}\bar{\bm{x}}-z^*} \right) \right\rceil
\end{equation}
iterations for a nondegenerate LP\@.
We also show that an upper bound for the number of iterations can be expressed as a form independent of the objective value, 
\begin{equation} \notag
  (n-m) \left\lceil m^{1+\frac{1}{p}} \frac{\gamma^2}{\delta^2} \log\left(\frac{m\gamma}{\delta}\right) \right\rceil.
\end{equation}
Finally, we prove that, for an LP formulation of the discounted Markov decision problem with a fixed discount factor, the simplex method with the $p$-norm rule is a strongly polynomial-time algorithm.

The rest of this paper is organized as follows.
Section~\ref{sec:Preliminaries} explains some preliminaries and previous research.
Section~\ref{sec:Steepest-pnorm-rules} proposes the $p$-norm rule as a generalization of the steepest-edge rule.
In Section~\ref{sec:Analysis}, we show two upper bounds for the number of iterations of the simplex method with the $p$-norm rule.
Section~\ref{sec:MDP} presents the application of our results to the discounted Markov decision problem.
In Section~\ref{sec:Discussion-Future-Work}, we summarize the result of the analysis and discuss our future work.

\section{Preliminaries and previous research} \label{sec:Preliminaries}

In this section, we first review linear programming and the simplex method.
Next, we define the most negative coefficient rule as a pivoting rule for the simplex method.
Finally, we explain some previous studies about the number of iterations of the simplex method with the most negative coefficient rule.

\subsection{Linear programming problem} \label{sec:LP}

We consider an LP with $n$~nonnegative variables and $m$~constraints:
\begin{equation} \label{eq:LP-Primal}
  \begin{alignedat}{2}
    &\text{minimize}   & \qquad & \bm{c}^{\top} \bm{x} \\
    &\text{subject to} &        & \bm{A} \bm{x} = \bm{b}, \\
    &                  &        & \bm{x} \ge \bm{0},
  \end{alignedat}
\end{equation}
where $\bm{A} \in \mathbb{R}^{m \times n}$, $\bm{b} \in \mathbb{R}^m$, and $\bm{c} \in \mathbb{R}^n$ are given data and $\bm{x} \in \mathbb{R}^n$ is a variable vector.
The dual problem of~\eqref{eq:LP-Primal} is expressed as follows:
\begin{equation} \label{eq:LP-Dual}
  \begin{alignedat}{2}
    &\text{maximize}   & \qquad & \bm{b}^{\top} \bm{y} \\
    &\text{subject to} &        & \bm{A}^{\top} \bm{y} + \bm{s} = \bm{c}, \\
    &                  &        & \bm{s} \ge \bm{0},
  \end{alignedat}
\end{equation}
where $\bm{y} \in \mathbb{R}^m$ and $\bm{s} \in \mathbb{R}^n$ are variable vectors.
The problem~\eqref{eq:LP-Primal} is called the primal problem with respect to the dual problem~\eqref{eq:LP-Dual}.

The duality theorem is a well-known relationship between a primal and a dual problem.
\begin{theorem}[Duality Theorem] \label{thm:Dual-Theorem}
  If the primal problem~\eqref{eq:LP-Primal} has an optimal solution, so does the dual problem~\eqref{eq:LP-Dual} and these optimal values are equal.
  In other words, if $\bm{x}^*$ is a primal optimal solution, there is a dual optimal solution $(\bm{y}^*,\bm{s}^*)$ such that
  \begin{equation} \notag
    \bm{c}^{\top} \bm{x}^* = \bm{b}^{\top} \bm{y}^*.
  \end{equation}
\end{theorem}\noindent
In addition, between a primal feasible solution $\bm{x}$ and a dual feasible solution $(\bm{y},\bm{s})$, the following relationship holds:
\begin{equation} \label{eq:Primal-Dual-Gap}
  \bm{c}^{\top} \bm{x} - \bm{b}^{\top} \bm{y} = \bm{x}^{\top} \bm{s}.
\end{equation}

Next, we define a dictionary.
Let $\{B, N\}$ be a partition of the index set of variables $\eset{1,2,\ldots,n}$.
We can split $\bm{A}$, $\bm{c}$, and~$\bm{x}$ corresponding to $\{B, N\}$ as follows:
\begin{equation} \notag
  \bm{A} = \begin{bmatrix}
    \bm{A}_B & \bm{A}_N
  \end{bmatrix},\quad
  \bm{c} = \begin{bmatrix}
    \bm{c}_B\\
    \bm{c}_N
  \end{bmatrix},\quad
  \bm{x} = \begin{bmatrix}
    \bm{x}_B\\
    \bm{x}_N
  \end{bmatrix}.
\end{equation}
By these splits, the problem~\eqref{eq:LP-Primal} is written as
\begin{equation} \label{eq:Primal-Separate}
  \begin{alignedat}{2}
    &\text{minimize}   & \qquad & \bm{c}_B^{\top} \bm{x}_B + \bm{c}_N^{\top} \bm{x}_N\\
    &\text{subject to} &        & \bm{A}_B \bm{x}_B + \bm{A}_N \bm{x}_N = \bm{b},\\
    &                  &        & \bm{x} = (\bm{x}_B, \bm{x}_N) \ge \bm{0}.
  \end{alignedat}
\end{equation}
If $|B| = m$ and $\bm{A}_B$ is nonsingular, $B$ and $N=\{1,2,\ldots,n\} \setminus B$ are called a basis and a nonbasis, respectively.

The problem~\eqref{eq:Primal-Separate} is transformed by multiplying the equality constraint by~$\bm{A}_B^{-1}$ from left:
\begin{equation} \label{eq:Primal-Dictionary}
  \begin{alignedat}{2}
    &\text{minimize}   & \qquad & z_0 + \bar{\bm{c}}_N^{\top} \bm{x}_N\\
    &\text{subject to} &        & \bm{x}_B = \bar{\bm{b}} - \bar{\bm{A}}_N \bm{x}_N,\\
    &                  &        & \bm{x} = (\bm{x}_B, \bm{x}_N) \ge \bm{0},
  \end{alignedat}
\end{equation}
where $z_0 = \bm{c}_B^{\top} \bm{A}_B^{-1} \bm{b}$, $\bar{\bm{c}}_N = \bm{c}_N - \bm{A}_N^{\top} (\bm{A}_B^{\top})^{-1} \bm{c}_B$, $\bar{\bm{b}} = \bm{A}_B^{-1} \bm{b}$, and $\bar{\bm{A}}_N = \bm{A}_B^{-1} \bm{A}_N$.
The form~\eqref{eq:Primal-Dictionary} is called a dictionary for a basis $B$.
The vector $\bar{\bm{c}}_N$ is called a reduced cost vector and $\bar{\bm{A}}_N$ is called a nonbasic matrix.
A basic solution is a solution $\bm{x}$ such that $(\bm{x}_B,\bm{x}_N) = (\bar{\bm{b}}, \bm{0})$ in the dictionary~\eqref{eq:Primal-Dictionary}.
The elements of $\bm{x}_B$ are basic variables and those of $\bm{x}_N$ are nonbasic variables.
A basic solution $\bm{x}$ is called a BFS (basic feasible solution) if $\bm{x}_B\ge\bm{0}$.
A basis~$B$ and a nonbasis~$N$ are respectively called a feasible basis and a feasible nonbasis if the corresponding basic solution is feasible.

\subsection{Simplex method} \label{sec:Simplex-Method}

For a given BFS in the dictionary~\eqref{eq:Primal-Dictionary}, if $\bar{\bm{c}}_N \ge \bm{0}$, then the BFS is optimal;
otherwise, increasing the value of a nonbasic variable with a negative reduced cost decreases the objective value.
We explain this procedure---the simplex method.

We use the following notations unless otherwise stated:
\begin{equation} \notag
  \bm{x}_B = \begin{bmatrix}
    x_{i_1}\\
    x_{i_2}\\
    \vdots\\
    x_{i_m}
  \end{bmatrix},~
  \bm{x}_N = \begin{bmatrix}
    x_{j_1}\\
    x_{j_2}\\
    \vdots\\
    x_{j_\ell}
  \end{bmatrix},~
  \bar{\bm{b}} = \begin{bmatrix}
    \bar{b}_1\\
    \bar{b}_2\\
    \vdots\\
    \bar{b}_m
  \end{bmatrix},~
  \bar{\bm{c}}_N = \begin{bmatrix}
    \bar{c}_1\\
    \bar{c}_2\\
    \vdots\\
    \bar{c}_\ell\\
  \end{bmatrix},~
  \bar{\bm{A}}_N = \begin{bmatrix}
    \bar{a}_{11} & \bar{a}_{12} & \cdots & \bar{a}_{1\ell}\\
    \bar{a}_{21} & \bar{a}_{22} & \cdots & \bar{a}_{2\ell}\\
    \vdots       & \vdots       & \ddots & \vdots\\
    \bar{a}_{m1} & \bar{a}_{m2} & \cdots & \bar{a}_{m\ell}
  \end{bmatrix},
\end{equation}
where $\ell = n-m$.
By using this notation, the dictionary~\eqref{eq:Primal-Dictionary} is expressed as follows:
\begin{equation} \label{eq:Primal-Dictionary-Plain}
  \newlength{\len}
  \settowidth{\len}{ $x_{i_1}, x_{i_2}, \ldots, x_{i_m}, x_{j_1}, x_{j_2}, \ldots, x_{j_\ell} \ge 0.$ }
  \begin{alignedat}{12}
    &\text{minimize}   & \qquad         &       & z_0       & {}+{} & \bar{c}_1    x_{j_1} & {}+{} & \cdots & {}+{} & \bar{c}_s    x_{j_s} & {}+{} & \cdots & {}+{} & \bar{c}_{\ell}  x_{j_{\ell}} \phantom{,} \\
    &\text{subject to} & \qquad x_{i_1} & {}={} & \bar{b}_1 & {}-{} & \bar{a}_{11} x_{j_1} & {}-{} & \cdots & {}-{} & \bar{a}_{1s} x_{j_s} & {}-{} & \cdots & {}-{} & \bar{a}_{1\ell} x_{j_{\ell}}, \\
    &                  &                &       &           &       &                      &       & \vdots &       &                      &       &        &       & \\
    &                  & \qquad x_{i_r} & {}={} & \bar{b}_r & {}-{} & \bar{a}_{r1} x_{j_1} & {}-{} & \cdots & {}-{} & \bar{a}_{rs} x_{j_s} & {}-{} & \cdots & {}-{} & \bar{a}_{r\ell} x_{j_{\ell}}, \\
    &                  &                &       &           &       &                      &       & \vdots &       &                      &       &        &       & \\
    &                  & \qquad x_{i_m} & {}={} & \bar{b}_m & {}-{} & \bar{a}_{m1} x_{j_1} & {}-{} & \cdots & {}-{} & \bar{a}_{ms} x_{j_s} & {}-{} & \cdots & {}-{} & \bar{a}_{m\ell} x_{j_{\ell}},\\
    &                  & x_{i_1}, x_{i_2}, \ldots, x_{i_m}, x_{j_1}, x_{j_2}, \ldots, x_{j_\ell} \ge 0. \hspace{-\len} \hphantom{x_{i_m}}
  \end{alignedat}
\end{equation}

Let $x_{j_s}$ be a nonbasic variable having a negative reduced cost~$\bar{c}_s$.
By increasing the value of $x_{j_s}$ from $0$ to $\theta_s > 0$, the objective value decreases by $-\bar{c}_s \theta_s > 0$ and the value of each basic variable $x_{i_k}$ changes from~$\bar{b}_k$ to~$\bar{b}_k-\bar{a}_{ks} \theta_s$.
Set $\theta_s$ as follows:
\begin{equation} \label{eq:minimum-ratio}
  \theta_s = \min\iset{\frac{\bar{b}_k}{\bar{a}_{ks}}}{\bar{a}_{ks}>0;~ k=1,2,\ldots,m}.
\end{equation}
Let $r$ be an index $k \in \eset{1,2,\ldots,m}$ attaining the minimum in the equality~\eqref{eq:minimum-ratio}.
By the definition of $\theta_s$ and~$r$, the equality $x_{i_r} - \bar{a}_{rs} \theta_s = 0$ holds.
The basic variable~$x_{i_r}$ therefore leaves a basis (enters a nonbasis), and the nonbasic variable~$x_{j_s}$ accordingly enters a basis.
The variable changed from a basis to a nonbasis is called a leaving variable, and the one changed from a nonbasis to a basis is called an entering variable.
In such a way, a basic variable is exchanged with a nonbasic variable in each iteration of the simplex method.

A rule to choose an entering variable from nonbasic variables is called a pivoting rule.
Many pivoting rules have been proposed because they have great influence on the number of iterations of the simplex method.
In this section, we describe the most negative coefficient rule, i.e., Dantzig's original one, and the best improvement rule.

The most negative coefficient rule chooses a nonbasic variable with the smallest reduced cost as an entering variable.
That is, this rule focuses on a variable that maximizes the decrease in the objective value per unit increase.
In other words, for a given feasible dictionary~\eqref{eq:Primal-Dictionary}, the rule chooses an entering variable $x_{j_d}$ such that
\begin{equation} \notag
  d \in \argmin_{k} \iset{\bar{c}_k}{k=1,2,\ldots,\ell}.
\end{equation}

The best improvement rule is to choose, in each iteration, a nonbasic variable that decreases the objective value most.
The simplex method with the best improvement rule needs fewer simplex iterations than that with the most negative coefficient rule in practice~\cite{Ploskas-Samaras-2014}; however, the simplex method with the best improvement rule is an exponential-time algorithm in the worst case~\cite{Jeroslow-1973}, as well as that with the most negative coefficient rule~\cite{Klee-Minty-1972}.

\subsection{Previous research} \label{sec:Previous-Research}

Although a number of pivoting rules have been proposed, it is still an open problem whether there exists a pivoting rule that can solve any LP in polynomial time.
Recently, some research paid attention to the number of different BFSs generated by the simplex method to consider the complexity.

Kitahara and Mizuno~\cite{Kitahara-Mizuno-2013-a} showed an upper bound for the number of different BFSs generated by the simplex method with the most negative coefficient rule.
They proved that the number of different BFSs generated by the simplex method is at most
\begin{equation} \label{eq:bound-by-Kitahara-1}
  \left\lceil \frac{m\gamma}{\delta} \log\left( \frac{\bm{c}^{\top} \bm{x}^0 - z^*}{\bm{c}^{\top} \bar{\bm{x}} - z^*} \right) \right\rceil
\end{equation}
for a standard LP with $n$~variables and $m$~constraints (see Section~\ref{sec:Introduction} or Table~\ref{tbl:Notation} in Section~\ref{sec:Assumption-Notation} for the notations in \eqref{eq:bound-by-Kitahara-1} and~\eqref{eq:bound-by-Kitahara-2}).
Moreover, another upper bound independent of the objective value is given by
\begin{equation} \label{eq:bound-by-Kitahara-2}
  (n-m) \left\lceil \frac{m\gamma}{\delta} \log\left( \frac{m\gamma}{\delta} \right) \right\rceil.
\end{equation}
Note that, for a nondegenerate LP the upper bounds \eqref{eq:bound-by-Kitahara-1} and \eqref{eq:bound-by-Kitahara-2} for the number of different BFSs can be regarded as upper bounds for the number of iterations of the simplex method with the most negative coefficient rule.
In addition, they showed that, for a nondegenerate LP the bounds~\eqref{eq:bound-by-Kitahara-1} and~\eqref{eq:bound-by-Kitahara-2} are also upper bounds for the number of iterations of the simplex method with the best improvement rule.

Kitahara and Mizuno~\cite{Kitahara-Mizuno-2013-b} also studied any pivoting rule that does not increase the objective value in each iteration.
Using such a pivoting rule, the simplex method generates different BFSs at most
\begin{equation} \label{eq:bound-by-Kitahara-3}
  \left\lceil \min\eset{m,n-m} \frac{\gamma_P \gamma_D^{\prime}}{\delta_P \delta_D^{\prime}} \right\rceil,
\end{equation}
where $\gamma_P$ and~$\delta_P$ are respectively the maximum and minimum positive elements in all BFSs, and $\gamma_D^{\prime}$ and~$\delta_D^{\prime}$ are respectively the maximum and minimum absolute values of negative reduced costs in all BFSs.
On nondegeneracy assumption of LPs, the upper bound \eqref{eq:bound-by-Kitahara-3} for the number of different BFSs can be regarded as that for the number of simplex iterations.

\section{$p$-norm rule: generalization of steepest-edge}\label{sec:Steepest-pnorm-rules}

We explain the steepest-edge rule in Section~\ref{sec:steepest-edge-rule}, and then propose a $p$-norm rule in Section~\ref{sec:p-norm-rule}, which is a generalization of the steepest-edge rule.

\subsection{Steepest-edge rule} \label{sec:steepest-edge-rule}

According to Forrest and Goldfarb~\cite{Forrest-Goldfarb-1992},  efficiency of the steepest-edge rule was reported for the first time in the computational experiment by Kuhn and Quandt~\cite{Kuhn-Quandt-1963} and that by Wolfe and Cutler~\cite{Wolfe-Cutler-1963}; several computational experiments showed that the steepest-edge rule needs fewer simplex iterations than the most negative coefficient rule (e.g., \cite{Goldfarb-Reid-1977,Ploskas-Samaras-2014}) and than the best improvement rule (e.g., \cite{Ploskas-Samaras-2014}).

The steepest-edge rule considers the amount of the change of all variables in contrast to the most negative coefficient rule.
That is, this rule focuses on the difference vector of basic solutions and the decrease in the objective value.
When a nonbasic variable~$x_{j_k}$ increases by $\theta_k > 0$ in a feasible dictionary~\eqref{eq:Primal-Dictionary}, the objective value decreases by $-\bar{c}_k \theta_k > 0$ and the value of each basic variable~$x_{i_u}$ changes from~$\bar{b}_u$ to $\bar{b}_u -\bar{a}_{uk} \theta_k$.
Hence, the norm of the difference vector of the two solutions before and after increasing~$x_{j_k}$ is
\begin{equation} \notag
  \sqrt{ \theta_k^2 + \sum_{i=1}^{m} (\bar{a}_{ik} \theta_k)^2 } = \theta_k \sqrt{ 1 + \sum_{i=1}^m \bar{a}_{ik}^2 }~. 
\end{equation}
Therefore the decrease in the objective value per unit length of the difference vector is expressed as
\begin{equation} \label{eq:Objective-Norm-Ratio}
  \frac{-\bar{c}_k \theta_k}{\theta_k \sqrt{\getheight 1 + \sum_{i=1}^m \bar{a}_{ik}^2 }} = \frac{-\bar{c}_k}{\sqrt{\getheight 1 + \sum_{i=1}^m \bar{a}_{ik}^2 }}.
\end{equation}
The steepest-edge rule chooses a nonbasic variable that maximizes the above decrease.
In other words, a nonbasic variable $x_{j_s}$ is chosen as an entering variable such that the following is satisfied:
\begin{equation} \notag
  s \in \argmin_{k} \iset{ \frac{\bar{c}_k}{\sqrt{\getheight 1 + \sum_{i=1}^{m} \bar{a}_{ik}^2}} }{ k=1,2,\ldots,\ell }.
\end{equation}

\subsection{$p$-norm rule}\label{sec:p-norm-rule}

Let $\left(\bm{v}_N\right)_k$ be $k$-th column vector of an $n \times \ell$ matrix~$\bm{V}_N$, where 
\begin{equation}
  \bm{V}_N = \begin{bmatrix}
    - \bar{\bm{A}}_N \\
    \bm{I}
  \end{bmatrix}
\end{equation}
and $\bm{I}$ is the $\ell \times \ell$ identity matrix.
By this notation, the condition that an entering variable~$x_{j_s}$ must satisfy for the steepest-edge rule is written as
\begin{equation} \label{eq:steepest-Edge-Definition}
  s \in \argmin_{k} \iset{ \frac{\bar{c}_k}{\getheight\norm[2]{\left(\bm{v}_N\right)_k}} }{ k=1,2,\ldots,\ell }.
\end{equation}

Here we propose a $p$-norm rule as a pivoting rule for the simplex method.
This rule is a generalized steepest-edge where the norm is changed from 2-norm to $p$-norm in~\eqref{eq:steepest-Edge-Definition}.
The $p$-norm rule selects a nonbasic variable~$x_{j_s}$ that satisfies
\begin{equation} \notag
  s \in \argmin_{k} \iset{ \frac{\bar{c}_k}{\getheight\norm[p]{\left(\bm{v}_N\right)_k}} }{ k=1,2,\ldots,\ell }
\end{equation}
as an entering variable.

\section{Analysis of the number of iterations} \label{sec:Analysis}

In this section, we show upper bounds for the number of iterations of the simplex method with the $p$-norm rule.
When $p=2$, these upper bounds are ones for the steepest-edge rule.

\subsection{Assumptions and notations} \label{sec:Assumption-Notation}

For our analysis, we assume the following:
\begin{itemize}
  \item $\text{rank~} \bm{A} = m$;
  \item An initial BFS $\bm{x}^0$ is available;
  \item The problems~\eqref{eq:LP-Primal} and~\eqref{eq:LP-Dual} have optimal basic solutions, denoted by $\bm{x}^*$ and $(\bm{y}^*, \bm{s}^*)$, respectively, and these optimal values are~$z^*$;
  \item An initial BFS $\bm{x}^0$ is not optimal, i.e., its objective value is larger than~$z^*$;
  \item The problem~\eqref{eq:LP-Primal} is nondegenerate, i.e., each basic variable has a positive value in any BFS\@.
\end{itemize}
The first three assumptions are the same as the previous study~\cite{Kitahara-Mizuno-2013-a}.
Moreover, the fourth assumption was also imposed in~\cite{Kitahara-Mizuno-2013-a} implicitly.

In addition to the notations defined earlier, let $\mathcal{B}$ and~$\mathcal{N}$ be the set of all feasible bases and all feasible nonbases, respectively.
We summarize the notations in Table~\ref{tbl:Notation}.
\begin{table}
  \centering
  \caption{Notations}
  \begin{tabular}{ll}
    \hline
    $m$ & the number of constraints\\
    $n$ & the number of variables\\
    $\ell$ & the constant equal to~$n-m$ \\
    $\gamma$ & the maximum positive element in all BFSs\\
    $\delta$ & the minimum positive element in all BFSs\\
    $z^*$ & the optimal value\\
    $\bm{x}^0$ & an initial BFS\\
    $\bar{\bm{x}}$ & a second optimal BFS, i.e.,  a BFS with the second smallest objective value\\
    $\left(\bm{v}_N\right)_k$ & $k$-th column vector of an $n \times \ell$ matrix 
    $\bm{V}_N = \begin{bmatrix}
      - \bar{\bm{A}}_N \\
      \bm{I}
    \end{bmatrix}$. \\ 
    $\mathcal{B}$ & the set consisting of all feasible bases\\
    $\mathcal{N}$ & the set consisting of all feasible nonbases, i.e., \\
                  & $\mathcal{N} = \iset{ N }{ B \in \mathcal{B},~N = \eset{1,2,\ldots,n}\setminus B }$  \\ \hline
  \end{tabular}
  \label{tbl:Notation}
\end{table}

\subsection{Upper bound dependent on the second optimal value} \label{sec:Upper-Bound-Second-Opt}

We now start our analysis of the $p$-norm rule from the following lemma about a lower bound for the optimal value.
\begin{lemma}[Kitahara and Mizuno~\cite{Kitahara-Mizuno-2013-a}] \label{lem:Lower-Bound-Of-Optimal-Value}
  Let $\bm{x}^t$ be the $t$-th solution generated by the simplex method with the most negative coefficient rule and let $B^t$ and~$N^t$ be the corresponding basis and nonbasis to~$\bm{x}^t$, respectively.
Moreover, set $\Delta_d^t = -\min\iset{\bar{c}_j}{j=1,2,\ldots,\ell}$.
  Then we have
  \begin{equation} \label{eq:Lower-Bound-Of-Optimal-Value}
    z^* \ge \bm{c}^{\top} \bm{x}^t - m \gamma \Delta_d^t.
  \end{equation}
\end{lemma}
\begin{proof}
  Let $\bm{x}^*$ be a basic optimal solution of the problem~\eqref{eq:LP-Primal}.
  Then we obtain
  \begin{equation} \notag
    \begin{alignedat}{1}
      z^* &= \bm{c}^{\top} \bm{x}^{*}\\
          &= \bm{c}_{B^t}^{\top} \bm{A}_{B^t}^{-1} \bm{b} + \bar{\bm{c}}_{N^t}^{\top} \bm{x}_{N^t}^* \\
          &\ge \bm{c}^{\top} \bm{x}^t - \Delta_d^t \bm{e}^{\top} \bm{x}_{N^t}^* \\
          &\ge \bm{c}^{\top} \bm{x}^t - m \gamma \Delta_d^t.
    \end{alignedat}
  \end{equation}
  The second inequality holds since $\bm{x}^*$ has $m$ positive elements and each element is bounded above by~$\gamma$.
  Thus, we have the inequality~\eqref{eq:Lower-Bound-Of-Optimal-Value}.
\end{proof}\noindent
This lemma was proven by Kitahara and Mizuno~\cite{Kitahara-Mizuno-2013-a}.
In their paper, they analyzed the simplex method with the most negative coefficient rule.
However, as above the proof is not based on a specific property of any pivoting rule.
Accordingly, Lemma~\ref{lem:Lower-Bound-Of-Optimal-Value} also holds for the simplex method with the $p$-norm rule.

Assume that a feasible dictionary~\eqref{eq:Primal-Dictionary} is given at the $t$-th iteration of the simplex method.
Let $x_{j_s}$ and~$x_{j_d}$ be entering variables chosen by the $p$-norm and most negative coefficient rules, respectively.
Then, by the definition of the $p$-norm rule, we have
\begin{equation} \label{eq:Reduced-Cost-Inequality}
  \frac{ \bar{c}_s }{ \getheight\norm[p]{\left(\bm{v}_N\right)_s} } \le \frac{ \bar{c}_d }{ \getheight\norm[p]{\left(\bm{v}_N\right)_d} }~.
\end{equation}
Due to the inequality~\eqref{eq:Reduced-Cost-Inequality}, we obtain
\begin{equation} \label{eq:Delta-Inequality}
  \Delta_s \ge \Delta_d \frac{ \getheight\norm[p]{\left(\bm{v}_N\right)_s} }{ \getheight\norm[p]{\left(\bm{v}_N\right)_d} },
\end{equation}
where $\Delta_s = -\bar{c}_s$ and $\Delta_d = -\bar{c}_d$.
We represent $\norm[p]{\left(\bm{v}_N\right)_s}/\norm[p]{\left(\bm{v}_N\right)_d}$ as~$\ratio_N$ and set $\ratio = \min\iset{\ratio_N}{N \in \mathcal{N}}$.
Then,
\begin{equation} \label{eq:Ratio-Inequality}
  \Delta_s \ge \ratio \Delta_d
\end{equation}
holds for all feasible nonbases.
We will analyze the detail of $\ratio$ in Section~\ref{sec:Lower-Bound-Ratio}.

Next, we show that, in each iteration of the simplex method with the $p$-norm rule, the difference between the objective value and the optimal value decreases at a constant ratio or more.

\begin{lemma} \label{lem:Gap-Rate}
  Let $\bm{x}^t$ and $\bm{x}^{t+1}$ be the $t$-th and $(t+1)$-th solutions of the simplex method with the $p$-norm rule, respectively.
  Then the following inequality holds:
  \begin{equation} \label{eq:Gap-Rate}
    \bm{c}^{\top} \bm{x}^{t+1} - z^* \le \left( 1 - \frac{\ratio \delta}{m \gamma} \right) \left( \bm{c}^{\top} \bm{x}^{t} - z^* \right).
  \end{equation}
\end{lemma}
\begin{proof}
  The objective value decreases by $\Delta_s x_{j_s}^{t+1}$ at the iteration.
  Moreover, by the definition of $\delta$ and $\gamma$, 
  \begin{equation} \notag
    x_j > 0 ~\Longrightarrow~ \delta \le x_j \le \gamma \quad (j=1,2,\ldots,n)
  \end{equation}
  holds for any BFS $\bm{x}$.
  Thus, we obtain the following inequality:
  \begin{equation} \notag
    \bm{c}^{\top} \bm{x}^{t} - \bm{c}^{\top} \bm{x}^{t+1} = \Delta_s x_{j_s}^{t+1} \ge \Delta_s \delta.
  \end{equation}
  By the inequalities~\eqref{eq:Lower-Bound-Of-Optimal-Value} and~\eqref{eq:Ratio-Inequality}, we have
  \begin{equation} \notag
    \Delta_s \delta \ge \ratio \delta \Delta_d \ge \ratio \delta \cdot \frac{\bm{c}^{\top} \bm{x}^t - z^*}{m\gamma}.
  \end{equation}
  Hence, we obtain
  \begin{equation} \notag
    \bm{c}^{\top} \bm{x}^{t} - \bm{c}^{\top} \bm{x}^{t+1} \ge \frac{\ratio\delta}{m\gamma} \left( \bm{c}^{\top} \bm{x}^t - z^* \right),
  \end{equation}
which leads to the inequality~\eqref{eq:Gap-Rate}.
\end{proof}

Applying the inequality~\eqref{eq:Gap-Rate} from $t=0,1,2,\ldots$ in order, we have
\begin{equation} \label{eq:Gap-Rate-Initial}
  \bm{c}^{\top} \bm{x}^{t} - z^* \le \left( 1 - \frac{\ratio\delta}{m\gamma} \right)^t \left( \bm{c}^{\top} \bm{x}^0 - z^* \right).
\end{equation}
Let $\bar{\bm{x}}$ be a second optimal BFS, that is, whose objective value is the smallest except that of the optimal solution.
If the $t$-th solution $\bm{x}^{t}$ satisfies the following inequality, it is optimal:
\begin{equation} \notag
  \bm{c}^{\top} \bm{x}^t - z^* < \bm{c}^{\top} \bar{\bm{x}} - z^*.
\end{equation}
The simplex method with the $p$-norm rule therefore finds an optimal solution and terminates after $T$~iterations starting from an initial BFS~$\bm{x}^0$, where $T$ is the smallest integer~$t$ such that the right-hand side value in the inequality~\eqref{eq:Gap-Rate-Initial} is less than $\bm{c}^{\top} \bar{\bm{x}} - z^*$.
As discussed above, we have the following theorem.
\begin{theorem} \label{thm:Upper-Bound-With-Second-Optimal-Ratio}
  Let $\bar{\bm{x}}$ be a BFS of the problem~\eqref{eq:LP-Primal} whose objective value is the second smallest.
  The simplex method with the $p$-norm rule finds an optimal solution in at most 
  \begin{equation} \label{eq:Upper-Bound-With-Second-Optimal-Ratio}
    \left\lceil \frac{m\gamma}{\ratio\delta} \log\left( \frac{\getheight\bm{c}^{\top}\bm{x}^0-z^*}{\getheight\bm{c}^{\top}\bar{\bm{x}}-z^*} \right) \right\rceil
  \end{equation}
  iterations for the problem~\eqref{eq:LP-Primal}.
\end{theorem}
\begin{proof}
  As mentioned earlier, the smallest integer~$t$ satisfying the following inequality is an upper bound for the number of iterations:
  \begin{equation} \notag
    \left( 1 - \frac{\ratio\delta}{m\gamma} \right)^t \left( \bm{c}^{\top} \bm{x}^0 - z^* \right) < \bm{c}^{\top} \bar{\bm{x}} - z^*.
  \end{equation}
  Solving this inequality for $t$, we have
  \begin{equation}
    t > \frac{\displaystyle \log\left(\frac{\getheight\bm{c}^{\top}\bm{x}^0-z^*}{\getheight\bm{c}^{\top}\bar{\bm{x}}-z^*}\right) }{\displaystyle -\log\left(1-\frac{\ratio\delta}{m\gamma}\right) }.
  \end{equation}
  From $\displaystyle \frac{1}{x} > \frac{1}{-\log\left(1-x\right)}$ holding in $0 < x < 1$, we obtain
  \begin{equation} \notag
    \frac{m\gamma}{\ratio\delta} > -\frac{1}{\displaystyle\log\left(1-\frac{\ratio\delta}{m\gamma}\right)},
  \end{equation}
  which leads to the theorem.
\end{proof}

\subsection{Upper bound independent of objective value} \label{sec:Upper-Bound-Objless}

The upper bound in Section~\ref{sec:Upper-Bound-Second-Opt} is dependent on the objective value.
We turn to obtain another upper bound independent of the objective value.

First we consider the following lemma.
\begin{lemma}[Kitahara and Mizuno~\cite{Kitahara-Mizuno-2013-a}] \label{lem:Non-Optimal-Variable-Upper-Bound}
  Let $\bm{x}^{t}$ and $B^t$ be the $t$-th solution of the simplex method and the corresponding basis to~$\bm{x}^{t}$, respectively.
  If $\bm{x}^t$ is not optimal, there exists $\bar{\jmath}\in B^t$ that satisfies the following conditions:
  \begin{equation} \label{eq:Optimal-Dual-Slack-Lower-Bound}
    x_{\bar{\jmath}}^t > 0 \qquad \text{and} \qquad s_{\bar{\jmath}}^* \ge \frac{1}{m x_{\bar{\jmath}}^t} \left( \bm{c}^{\top} \bm{x}^t - z^* \right),
  \end{equation}
  where $\bm{s}^*$ is the slack vector of an optimal basic solution of the dual problem~\eqref{eq:LP-Dual}.
  Furthermore, the $k$-th solution $\bm{x}^{k}$ satisfies
  \begin{equation} \label{eq:Non-Optimal-Variable-Upper-Bound}
    x_{\bar{\jmath}}^{k} \le m x_{\bar{\jmath}}^t \, \frac{\bm{c}^{\top} \bm{x}^{k} - z^*}{\bm{c}^{\top} \bm{x}^t - z^*}
  \end{equation}
  for arbitrary positive integer $k$.
\end{lemma}
\begin{proof}
  We first prove the former.
  From the equation~\eqref{eq:Primal-Dual-Gap}, we have
  \begin{equation} \notag
    \bm{c}^{\top} \bm{x}^t - z^* = \bm{c}^{\top} \bm{x}^t - \bm{b}^{\top} \bm{y}^* = (\bm{x}^{t})^{\top} \bm{s}^* = \sum_{j\in B^t} x_j^t s_j^*.
  \end{equation}
  Since $\bm{x}^t \ge \bm{0}$, $\bm{s}^* \ge \bm{0}$, and $|B^t|=m$, there exists $\bar{\jmath}\in B^t$ such that
  \begin{equation} \notag
    x_{\bar{\jmath}}^t s_{\bar{\jmath}}^* \ge \frac{1}{m} \left(\bm{c}^{\top} \bm{x}^t - z^* \right).
  \end{equation}
  Here $\bm{x}^t$ is not optimal and thus the right-hand side is positive.
  Hence, $x_{\bar{\jmath}}^t > 0$, which shows the existence of a variable satisfying the conditions~\eqref{eq:Optimal-Dual-Slack-Lower-Bound}.

  Next we prove the latter．
  As shown above, for any positive integer $k$, 
  \begin{equation} \notag
    \bm{c}^{\top} \bm{x}^{k} - z^* = (\bm{x}^{k})^{\top} \bm{s}^{*} = \sum_{j=1}^n x_j^{k} s_j^*.
  \end{equation}
  In addition, $x_j^{k} \ge 0$ and $s_j^* \ge 0 \ (j=1,2,\ldots,n)$ hold, and thus we obtain
  \begin{equation} \notag
    \bm{c}^{\top} \bm{x}^{k} - z^* \ge x_{\bar{\jmath}}^{k} s_{\bar{\jmath}}^*.
  \end{equation}
  Using this and the inequality~\eqref{eq:Optimal-Dual-Slack-Lower-Bound}, we have
  \begin{equation} \notag
    x_{\bar{\jmath}}^{k} \le \frac{\bm{c}^{\top} \bm{x}^{k} - z^*}{s_{\bar{\jmath}}^*} \le m x_{\bar{\jmath}}^t \, \frac{\bm{c}^{\top} \bm{x}^{k} - z^*}{\bm{c}^{\top} \bm{x}^t - z^*}.
  \end{equation}
\end{proof}
Same as Lemma~\ref{lem:Lower-Bound-Of-Optimal-Value}, this proof is not based on a property of a specific pivoting rule.
Accordingly, Lemma~\ref{lem:Non-Optimal-Variable-Upper-Bound} also holds for the simplex method with the $p$-norm rule.

We now have an upper bound independent of the objective value.
\begin{theorem} \label{thm:Upper-Bound-Without-Objective-Function-Ratio}
  When applying the simplex method with the $p$-norm rule to the problem~\eqref{eq:LP-Primal}, the number of iterations is at most
  \begin{equation} \label{eq:Upper-Bound-Without-Objective-Function-Ratio}
    (n-m) \left\lceil \frac{m\gamma}{\ratio\delta} \log\left(\frac{m\gamma}{\delta}\right) \right\rceil.
  \end{equation}
\end{theorem}
\begin{proof}
  Let $r$ be an integer that is greater than or equal to~$1$.
  Moreover, let $\bm{x}^t$ and $\bm{x}^{t+r}$ be the $t$-th and $(t+r)$-th solutions of the simplex method, respectively.
  In addition, let $B^t$ be the corresponding basis of $\bm{x}^t$.
  Then, by Lemmas~\ref{lem:Gap-Rate} and \ref{lem:Non-Optimal-Variable-Upper-Bound}, and the definition of $\gamma$, there exists $\bar{\jmath}\in B^t$ such that 
  \begin{equation} \notag
    x_{\bar{\jmath}}^{t+r} \le m x_{\bar{\jmath}}^t \left( 1 - \frac{\ratio\delta}{m\gamma} \right)^r \le m \gamma \left( 1 - \frac{\ratio\delta}{m\gamma} \right)^r.
  \end{equation}
  Thus, when $r \geq (m\gamma)/(\ratio\delta) \cdot \log\left(m\gamma/\delta\right)$, the rightmost term is less than $\delta$, and $x_{\bar{\jmath}}^{t+r}$ is fixed to~$0$ by the definition of~$\delta$.

  If an optimal solution is not obtained after $r$~iterations satisfying the above inequality, we can apply the same procedure again.
  The number of variables that can be chosen as a basis decreases by one through each process.
  Due to the nondegeneracy assumption, the number of positive elements in any basic feasible solution is $m$, and thus this process occurs at most $n-m$ times.
  Hence, we have the desired result.
\end{proof}

\subsection{Lower bound for $\ratio$} \label{sec:Lower-Bound-Ratio}

The ratio~$\ratio$ is contained in the upper bounds in Theorems~\ref{thm:Upper-Bound-With-Second-Optimal-Ratio} and~\ref{thm:Upper-Bound-Without-Objective-Function-Ratio}.
In this section, we analyze a lower bound for~$\ratio$ to make these upper bounds clearer.

The definition of $\ratio$ is as follows:
\begin{equation} \notag
  \ratio = \min\iset{q_{N}}{N\in\mathcal{N}}, \quad \ratio_N = \frac{ \getheight\norm[p]{\left(\bm{v}_N\right)_s} }{ \getheight\norm[p]{\left(\bm{v}_N\right)_d} }.
\end{equation}
Hence, a lower bound for $\ratio$ can be derived from a lower bound of $\norm[p]{\left(\bm{v}_N\right)_s}$ divided by an upper bound of~$\norm[p]{\left(\bm{v}_N\right)_d}$.

Let $x_{j_k}$ and~$x_{i_r}$ be entering and leaving variables when the solution $\bm{x}^t$ changes to~$\bm{x}^{t+1}$ at the $t$-th iteration of the simplex method, respectively.
Set $\bm{w} = \bm{x}^{t+1}-\bm{x}^t$.
Since nonbasic variables except $x_{j_k}$ is unchanged through the iteration, we have
\begin{equation} \notag
  \norm[p]{\bm{w}}^p = |w_{i_1}|^p + |w_{i_2}|^p + \cdots + |w_{i_m}|^p + |w_{j_k}|^p.
\end{equation}
By the definition of~$\gamma$ and the nonnegative constraints of variables, any element of $\bm{w}$ is in the range $-\gamma$ to~$\gamma$.
That is, $|w_j| \le \gamma \ (j=1,2,\ldots,n)$ holds.
Thus, we obtain the following inequality:
\begin{equation} \notag
  |w_{j_k}|^p + |w_{i_r}|^p \le \norm[p]{\bm{w}}^p \le |w_{j_k}|^p + |w_{i_r}|^p + (m-1) \gamma^p.
\end{equation}
Furthermore, $\delta \le |w_{i_r}| \le \gamma$ holds due to the nondegeneracy assumption, we have
\begin{equation} \label{eq:X-Bounds-With-Xjk}
  |w_{j_k}|^p + \delta^p \le \norm[p]{\bm{w}}^p \le |w_{j_k}|^p + m \gamma^p.
\end{equation}
As mentioned in Section~\ref{sec:steepest-edge-rule}, each basic variable changes $- \bar{a}_{ik} \theta_k ~(i=1,2,\ldots,m)$ when $x_{j_k}$ increases by $\theta_k = |w_{j_k}|$.
Thus, the $p$-norm of the difference vector~$\bm{w}$ is
\begin{equation} \notag
  \norm[p]{\bm{w}} = \left( |w_{j_k}|^p + \sum_{i=1}^{m} |- \bar{a}_{ik} w_{j_k}|^p \right)^{1/p} = |w_{j_k}| \left( 1 + |\bar{a}_{ik}|^p \right)^{1/p}.
\end{equation}
Using the notation introduced in Section~\ref{sec:p-norm-rule}, the equality can be expressed~as
\begin{equation} \notag
  \norm[p]{\bm{w}} = |w_{j_k}| \cdot \norm[p]{\left(\bm{v}_N\right)_k}.
\end{equation}
Dividing the both sides of the inequality~\eqref{eq:X-Bounds-With-Xjk} by $|w_{j_k}|^p > 0$, we obtain
\begin{equation} \notag
  1 + \frac{\delta^p}{|w_{j_k}|^p} \le \norm[p]{\left(\bm{v}_N\right)_k}^p \le 1 + m \frac{\gamma^p}{|w_{j_k}|^p}.
\end{equation}
The relationship $\delta \le |w_{j_k}| \le \gamma$ gives upper and lower bounds for $\norm[p]{\left(\bm{v}_N\right)_k}^p$:
\begin{equation} \label{eq:Norm-Bounds}
  1 + \frac{\delta^p}{\gamma^p} \le \norm[p]{\left(\bm{v}_N\right)_k}^p \le 1 + m \frac{\gamma^p}{\delta^p}.
\end{equation}

Let $x_{j_d}$ and~$x_{j_s}$ be nonbasic variables chosen by the most negative coefficient and $p$-norm rules, respectively.
From the inequality~\eqref{eq:Norm-Bounds}, we have 
\begin{equation} \notag
  \frac{ \getheight\norm[p]{\left(\bm{v}_N\right)_s}^p }{ \getheight\norm[p]{\left(\bm{v}_N\right)_d}^p } \ge \frac{1 + (\delta/\gamma)^p}{1 + m (\gamma/\delta)^p} \ge \frac{1 + (\delta/\gamma)^p}{m + m (\gamma/\delta)^p} \ge \frac{\delta^p}{m\gamma^p}.
\end{equation}
Therefore $(\delta/\gamma) m^{-1/p}$ is a lower bound for~$\ratio$.
Applying this bound to Theorems~\ref{thm:Upper-Bound-With-Second-Optimal-Ratio} and~\ref{thm:Upper-Bound-Without-Objective-Function-Ratio}, the following theorems are immediately obtained.
\begin{theorem} \label{thm:Iteration-Upper-Bound-1}
  The simplex method with the $p$-norm rule finds an optimal solution in at most
  \begin{equation}
    \left\lceil m^{1+\frac{1}{p}} \frac{\gamma^2}{\delta^2} \log\left( \frac{\getheight\bm{c}^{\top}\bm{x}^0 - z^*}{\getheight\bm{c}^{\top}\bar{\bm{x}}-z^*} \right) \right\rceil
  \end{equation}
  iterations for the problem~\eqref{eq:LP-Primal}.
\end{theorem}
\begin{theorem} \label{thm:Iteration-Upper-Bound-2}
  The simplex method with the $p$-norm rule finds an optimal solution in at most
  \begin{equation}
    (n-m) \left\lceil m^{1+\frac{1}{p}} \frac{\gamma^2}{\delta^2} \log\left(\frac{m\gamma}{\delta}\right) \right\rceil.
  \end{equation}
  iterations for the problem~\eqref{eq:LP-Primal}.
\end{theorem}

\section{Application to Markov decision problem} \label{sec:MDP}

Ye~\cite{Ye-2011} proved that the simplex method with the most negative coefficient rule is a strongly polynomial-time algorithm for solving the discounted Markov decision problem (DMDP) with a fixed discount factor; this result motivated the research on the number of different BFSs by Kitahara and Mizuno~\cite{Kitahara-Mizuno-2013-a,Kitahara-Mizuno-2013-b}, which we mentioned in Section~\ref{sec:Previous-Research}.

On the other hand, the simplex method with the smallest index rule takes exponential time to solve the DMDP regardless of discount factors~\cite{Melekopoglou-Condon-1994}.

These results imply that pivoting rules are crucial for the complexity to solve the DMDP\@.
Since the DMDP satisfies the assumptions in Section~\ref{sec:Assumption-Notation}, we can apply the upper bound in Theorem~\ref{thm:Iteration-Upper-Bound-2} to analyze the DMDP\@.

Here we prove that the simplex method with the $p$-norm rule is also a strongly polynomial-time algorithm for the DMDP with a fixed discount factor.
The LP formulation and some properties of the DMDP are based on those given by Ye~\cite{Ye-2011}.

The DMDP can be formulated as the following LP:
\begin{equation} \label{eq:DMDP-Formulation}
  \begin{alignedat}{2}
    &\text{minimize}   & \qquad & \bm{c}^{\top} \bm{x}\\
    &\text{subject to} & \qquad & \left(E - \theta P\right) \bm{x} = \bm{e},\\
    &                  &        & \bm{x} \ge \bm{0},
  \end{alignedat}
\end{equation}
where $\theta \in [0,1)$ is the discount factor, $\bm{e} \in \mathbb{R}^{m}$ is the vector of all ones,
$E \in \mathbb{R}^{m \times n}$ is a 0-1 matrix that indicates whether the action $j$ can be chosen in the state $i$, $\bm{c} \in \mathbb{R}^{n}$ is a vector calculated by immediate costs, and $P \in \mathbb{R}^{m \times n}$ consists of transition probabilities.

Let $\bm{x}$ be a BFS of the problem~\eqref{eq:DMDP-Formulation}.
Ye~\cite{Ye-2005} proved that if $x_j$ is a basic variable, then $x_j$ satisfies
\begin{equation}
  1 \le x_j \le \frac{m}{1-\theta}.
\end{equation}
This inequality implies that the minimum and maximum values of all the positive elements of BFSs are not less than~$1$ and not more than $m/(1-\theta)$, respectively; in other words, $\delta \ge 1$ and $\gamma \le m/(1-\theta)$ hold.
Moreover, the problem~\eqref{eq:DMDP-Formulation} is nondegenerate.
We therefore obtain the following theorem.
\begin{theorem} \label{thm:DMDP-Upper-Bound}
  Consider the DMDP with $m$~states, $n$~actions, and a discount factor~$\theta$ formulated as the problem~\eqref{eq:DMDP-Formulation}.
  The simplex method with the $p$-norm rule takes at most
  \begin{equation}
    (n-m) \left\lceil \frac{m^{3+\frac{1}{p}}}{(1-\theta)^2} \log\left( \frac{m^2}{1-\theta} \right) \right\rceil
  \end{equation}
  iterations to solve the problem~\eqref{eq:DMDP-Formulation}, and is a strongly polynomial-time algorithm for the DMDP with a fixed discount factor.
\end{theorem}

\section{Discussion and future work} \label{sec:Discussion-Future-Work}

In this paper, we proposed the $p$-norm rule as a pivoting rule for the simplex method, which is a generalization of the steepest-edge rule.
In addition, we showed two upper bounds for the number of iterations taken by the simplex method with the $p$-norm rule for a nondegenerate LP\@.
One of the upper bounds is expressed~as
\begin{equation} \notag
  \left\lceil m^{1+\frac{1}{p}} \frac{\gamma^2}{\delta^2} \log\left( \frac{\getheight\bm{c}^{\top}\bm{x}^0-z^*}{\getheight\bm{c}^{\top}\bar{\bm{x}}-z^*} \right) \right\rceil,
\end{equation}
which depends on the second optimal solution (Theorem~\ref{thm:Iteration-Upper-Bound-1}); the other is
\begin{equation} \notag
  (n-m) \left\lceil m^{1+\frac{1}{p}} \frac{\gamma^2}{\delta^2} \log\left(\frac{m\gamma}{\delta}\right) \right\rceil,
\end{equation}
which is independent of the objective value (Theorem~\ref{thm:Iteration-Upper-Bound-2}).

The discounted Markov decision problem satisfies the assumptions we need for analysis; our results proved that the simplex method with the $p$-norm rule is a strongly polynomial-time algorithm for solving the problem with a fixed discount factor.

There are some further directions for this study.
Firstly, it is known that the steepest-edge rule takes fewer iterations than the most negative coefficient rule  and than the best improvement rule by computational experiments.
However, the upper bounds obtained in this paper are larger than that for these rules,
\begin{equation} \notag
  (n-m) \left\lceil \frac{m\gamma}{\delta} \log\left( \frac{m\gamma}{\delta} \right) \right\rceil,
\end{equation}
obtained by Kitahara and Mizuno~\cite{Kitahara-Mizuno-2013-a}.
Thus, to find a better upper bound is a further direction.
Secondly, there remains a matter to be discussed whether the simplex method with $p$-norm rule is strongly polynomial-time for solving the discounted Markov decision problem regardless of discount factors.
Yet another future work is to remove the nondegeneracy assumption of LPs for the $p$-norm rule.

\section*{Acknowledgement}

This work was partially supported by JSPS KAKENHI Grant Numbers JP17K01246 and JP15K15941.


\begin{thebibliography}{99}
\bibitem{Dantzig-1963}
G.B. Dantzig, Linear programming and extensions. Princeton University Press, 1963.

\bibitem{Forrest-Goldfarb-1992}
J.J. Forrest and D. Goldfarb, Steepest-edge simplex algorithms for LP. Mathematical Programming 57 (1992), \mbox{341--374}.

\bibitem{Goldfarb-Reid-1977}
D. Goldfarb and J.K. Reid, A practicable steepest-edge simplex algorithm. Mathematical Programming 12 (1977), \mbox{361--371}.

\bibitem{Goldfarb-Sit-1979}
D. Goldfarb and W.Y. Sit, Worst case behavior of the steepest edge simplex method. Discrete Applied Mathematics 1~(1979), \mbox{277--285}.

\bibitem{Jeroslow-1973}
R.G. Jeroslow, The simplex algorithm with the pivot rule of maximizing criterion improvement. Discrete Mathematics 4-4 (1973), \mbox{367-377}.

\bibitem{Kitahara-Mizuno-2013-a}
T. Kitahara and S. Mizuno, A bound for the number of different basic feasible solutions generated by the simplex method. Mathematical Programming Series~A 137~(2013), \mbox{579--586}.

\bibitem{Kitahara-Mizuno-2013-b}
T. Kitahara and S. Mizuno, An upper bound for the number of different basic feasible solutions generated by the primal simplex method with any selection rule of entering variables. Asia-Pacific Journal of Operational Research \mbox{30-3}~(2013).

\bibitem{Klee-Minty-1972}
V. Klee and G.J. Minty, How good is the simplex algorithm. Inequalities 3~(1972), \mbox{159--175}.

\bibitem{Kuhn-Quandt-1963}
H.W. Kuhn and R.E. Quandt, An experimental study of the simplex method. Proceedings of Symposia in Applied Mathematics 15~(1963), \mbox{107--124}.

\bibitem{Melekopoglou-Condon-1994}
M. Melekopoglou and A. Condon, On the complexity of the policy improvement algorithm for Markov decision processes. ORSA Journal on Computing \mbox{6-2}~(1994), \mbox{188--192}.

\bibitem{Ploskas-Samaras-2014}
N. Ploskas and N. Samaras, Pivoting rules for the revised simplex algorithm. Yugoslav Journal of Operations Research \mbox{24-3}~(2014), \mbox{321--332}.

\bibitem{Terlaky-Zhang-1993}
T. Terlaky and S. Zhang, Pivot rules for linear programming: A survey on recent theoretical developments. Annals of Operations Research \mbox{46-1} (1993), \mbox{203--233}.

\bibitem{Wolfe-Cutler-1963}
P. Wolfe and L. Cutler, Experiments in linear programming. In: R.L. Graves and P. Wolfe (eds.), Recent Advances in Mathematical Programming, McGraw-Hill, 1963. 

\bibitem{Ye-2005}
Y. Ye, A new complexity result on solving the Markov decision problem. Mathematics of Operations Research \mbox{30-3}~(2005), \mbox{733--749}.

\bibitem{Ye-2011}
Y. Ye, The simplex and policy-iteration methods are strongly polynomial for the Markov decision problem with a fixed discount rate. Mathematics of Operations Research \mbox{36-4}~(2011), \mbox{593--603}.

\end{thebibliography}
\end{document}